\documentclass[12pt]{amsart}

\usepackage{color}
\usepackage{bbm}
\usepackage{graphicx}
\usepackage{amsfonts,amssymb,amsmath,amsthm,mathrsfs,enumerate,multicol,esint}

\usepackage[pagebackref,hypertexnames=false, colorlinks, citecolor=red, linkcolor=red]{hyperref} 
\usepackage[backrefs]{amsrefs}

\begin{document}
\baselineskip 17.5pt
\hfuzz=6pt

\newtheorem{theorem}{Theorem}[section]
\newtheorem{prop}[theorem]{Proposition}
\newtheorem{lemma}[theorem]{Lemma}
\newtheorem{definition}[theorem]{Definition}
\newtheorem{cor}[theorem]{Corollary}
\newtheorem{example}[theorem]{Example}
\newtheorem{remark}[theorem]{Remark}
\newcommand{\ra}{\rightarrow}
\renewcommand{\theequation}
{\thesection.\arabic{equation}}
\newcommand{\ccc}{{\mathcal C}}
\newcommand{\one}{1\hspace{-4.5pt}1}

\def\R{\mathbb R}
\def\C{\mathbb C}

\newcommand{\bmo}{\rm BMO}

\newcommand{\al}{\alpha}
\newcommand{\be}{\beta}
\newcommand{\ga}{\gamma}
\newcommand{\Ga}{\Gamma}
\newcommand{\de}{\delta}
\newcommand{\ben}{\beta_n}
\newcommand{\De}{\Delta}
\newcommand{\ve}{\varepsilon}
\newcommand{\ze}{\zeta}
\newcommand{\Th}{\chi}
\newcommand{\ka}{\kappa}
\newcommand{\la}{\lambda}
\newcommand{\laj}{\lambda_j}
\newcommand{\lak}{\lambda_k}
\newcommand{\La}{\Lambda}
\newcommand{\si}{\sigma}
\newcommand{\Si}{\Sigma}
\newcommand{\vp}{\varphi}
\newcommand{\om}{\omega}
\newcommand{\Om}{\Omega}

\newcommand{\cs}{\mathcal S}
\newcommand{\csrn}{\mathcal S (\rn)}
\newcommand{\cm}{\mathcal M}
\newcommand{\cb}{\mathcal B}
\newcommand{\ce}{\mathcal E}
\newcommand{\cd}{D}
\newcommand{\cdp}{D'}
\newcommand{\csp}{\mathcal S'}

\newcommand{\lab}{\label}
\newcommand{\med}{\textup{med}}
\newcommand{\pv}{\textup{p.v.}\,}
\newcommand{\loc}{\textup{loc}}
\newcommand{\intl}{\int\limits}
\newcommand{\intlrn}{\int\limits_{\rn}}
\newcommand{\intrn}{\int_{\rn}}
\newcommand{\iintl}{\iint\limits}
\newcommand{\dint}{\displaystyle\int}
\newcommand{\diint}{\displaystyle\iint}
\newcommand{\dintl}{\displaystyle\intl}
\newcommand{\diintl}{\displaystyle\iintl}
\newcommand{\liml}{\lim\limits}
\newcommand{\suml}{\sum\limits}
\newcommand{\supl}{\sup\limits}
\newcommand{\f}{\frac}
\newcommand{\df}{\displaystyle\frac}
\newcommand{\p}{\partial}
\newcommand{\Ar}{\textup{Arg}}
\newcommand{\abssig}{\widehat{|\si_0|}}
\newcommand{\abssigk}{\widehat{|\si_k|}}
\newcommand{\tT}{\tilde{T}}
\newcommand{\tV}{\tilde{V}}
\newcommand{\wt}{\widetilde}
\newcommand{\wh}{\widehat}
\newcommand{\lp}{L^{p}}
\newcommand{\nf}{\infty}
\newcommand{\rrr}{\mathbf R}
\newcommand{\li}{L^\infty}
\newcommand{\rn}{\mathbf R^n}
\newcommand{\tf}{\tfrac}
\newcommand{\zzz}{\mathbf Z}

\newcommand{\MM}{\mathcal{M}}

\newcommand{\contain}[1]{\left( #1 \right)}
\newcommand{\ContainA}[1]{( #1 )}
\newcommand{\ContainB}[1]{\bigl( #1 \bigr)}
\newcommand{\ContainC}[1]{\Bigl( #1 \Bigr)}
\newcommand{\ContainD}[1]{\biggl( #1 \biggr)}
\newcommand{\ContainE}[1]{\Biggl( #1 \Biggr)}

\allowdisplaybreaks

\title[]{A note on weak factorisation of Meyer-type Hardy space via Cauchy integral operator}

\thanks{{\it {\rm 2010} Mathematics Subject Classification:} Primary: 42B35, 42B25.}

\author{Yongsheng Han}
\address{Yongsheng Han, Department of Mathematics, Auburn University, Alabama, USA}
\email{hanyong@auburn.edu}

\author{Ji Li}
\address{Ji Li, Department of Mathematics, Macquarie University, Sydney}
\email{ji.li@mq.edu.au}

\author{Cristina Pereyra}
\address{Cristina Pereyra, Department of Mathematics, The University of New Mexico, Albuquerque, USA}
\email{crisp@math.unm.edu}

\author{Brett D. Wick}
\address{Brett D. Wick, Department of Mathematics\\
         Washington University - St. Louis\\
         St. Louis, MO 63130-4899 USA
         }
\email{wick@math.wustl.edu}

\begin{abstract}
This paper provides a weak factorization for the Meyer-type Hardy space $H^1_b(\mathbb{R})$,  and characterizations of its dual ${\rm BMO}_b(\mathbb{R})$ and its predual  ${\rm VMO}_b(\mathbb{R})$ via boundedness and compactness of a suitable commutator with  the Cauchy integral $\mathscr{C}_{\Gamma}$, respectively.
Here $b(x)=1+iA'(x)$ where $A'\in L^{\infty}(\mathbb{R})$, and  the Cauchy integral $\mathscr{C}_{\Gamma}$ is associated to the Lipschitz curve $\Gamma=\{x+iA(x)\, : \, x\in \mathbb{R}\}$. 

\end{abstract}

\maketitle


\bigskip

\section{Introduction and Statement of Main Results}
\setcounter{equation}{0}

Given a bounded function  $b:\mathbb{R}^n\to \mathbb{C}$ such that  ${\rm Re}\, b(x) \geq 1$ for all $x\in \mathbb{R}^n$,  the Meyer-type Hardy space  $H^1_b(\mathbb{R}^n)$  consists of those functions $f:\mathbb{R}^n\to \mathbb{C}$ such that the product $bf$ belongs to the Hardy space  $H^1(\mathbb{R}^n)$. The Meyer-type space of bounded mean oscillation, space ${\rm BMO}_b(\mathbb{R}^n) $ consists of all functions $\frak A: \mathbb{R}^n\to \mathbb{C}$ such  that 
$\frak{A}/b$ belongs to ${\rm BMO}(\mathbb{R}^n)$ and is the dual of $H^1_b(\mathbb{R}^n)$.
These spaces were introduced by Yves Meyer \cite[Chapter XI, Section 10, p. 358]{Me}, in dimension one  in connection to the study of the Cauchy integral  associated with a Lipschitz curve and the $T(b)$ theorem.

In this note we study the Meyer-type Hardy space $H^1_b(\mathbb{R})$ and its dual ${\rm BMO}_b(\mathbb{R})$ for $b(x)=1+iA'(x)$ where $A'\in L^{\infty}(\mathbb{R})$, via the Cauchy integral $\mathscr{C}_{\Gamma}$ associated to the Lipschitz curve $\Gamma=\{x+iA(x)\, : \, x\in \mathbb{R}\}$. We present a weak factorization of $H^1_b(\mathbb{R})$ in terms of the Cauchy integral $\mathscr{C}_{\Gamma}$. We also obtain a characterization of ${\rm BMO}_b(\mathbb{R})$ and of ${\rm VMO}_b(\mathbb{R})$,  the Meyer-type space of vanishing mean oscillation, via boundedness and  compactness of a suitable commutator with the Cauchy integral respectively.

The \emph{Cauchy integral associated with the Lipschitz
curve}~$\Gamma$ is the integral operator $\mathscr{C}_{\Gamma}$ given by
\[
  \mathscr{C}_{\Gamma}(f)(x)
  := {\rm p.v.} \frac{1}{\pi i}\int_{\R} \frac{(1 + iA'(y))f(y)}{y-x + i(A(y) - A(x))}\,dy,
\]
where $f \in C_{c}^{\infty}(\R)$. Note that the Cauchy integral associated with the Lipschitz curve $\Gamma$ is not a standard Calder\'on-Zygmund operator because it lacks smoothness.    
The~$L^p$-boundedness
of~$\mathscr{C}_{\Gamma}$ is equivalent to that of the
related operator $\widetilde{\mathscr{C}}_{\Gamma}$ defined by
\[
  \widetilde{\mathscr{C}}_{\Gamma}(f)(x)
  := {\rm p.v.} \frac{1}{\pi i}\int_{\R} \frac{f(y)}{y-x + i(A(y) - A(x))}\,dy,
\]
Moreover,  the kernel
of~$\widetilde{\mathscr{C}}_{\Gamma}$ satisfies standard  size and smoothness estimates \cite{LNWW} and is therefore bounded on $L^p(\mathbb{R})$ for $p\in (1,\infty)$. Note that in \cite{LNWW} the operator $\mathscr{C}_{\Gamma}$ was denoted $\widetilde{\mathscr{C}}_{\Gamma}$ and viceversa. Hence
while~$\mathscr{C}_{\Gamma}(f)$  is initially defined for~$f \in
C_{c}^{\infty}(\R)$, the operator~$\mathscr{C}_{\Gamma}$ can be extended
to all~$f \in L^{p}(\R)$, for each~$p \in (1,\infty)$.

The related operator $\widetilde{\mathscr{C}}_{\Gamma}$ and its commutator with functions in ${\rm BMO}(\mathbb{R})$ were studied by Li, Nguyen, Ward, and Wick  in \cite{LNWW}. In this setting, one could appeal to a weak factorization for $H^1(\mathbb{R}^n)$ in terms of multilinear  Calder\'on-Zygmund operators,  due to  Li and Wick  \cite[Theorem 1.3]{LW},
to obtain the desired characterization of ${\rm BMO}(\mathbb{R})$ via boundedness of the commutator, and of ${\rm VMO}(\mathbb{R})$ 
via compactness of the commutator.

We want to study  the Meyer-type Hardy space, bounded mean oscillation space,  and vanishing mean oscillation space: $H^1_b(\mathbb{R})$, ${\rm BMO}_b(\mathbb{R})$, and ${\rm VMO}_b(\mathbb{R})$,  via the rougher operator $\mathscr{C}_{\Gamma}$. 
As it turns out, we can derive these results from the results for the related Cauchy integral operator \cite{LNWW}. Nevertheless we also present a direct constructive proof of the weak factorization valid for $H^1_b(\mathbb{R})$ that maybe of independent interest.

We now state  our main results.
For $b(x)= 1+i A'(x)$, we introduce the bilinear form associated with $b$ as follows:
 \begin{equation}\label{bilinear-form}
 \Pi_b(g,h)(x):=\frac{1}{b(x)} \big (g(x)\cdot \mathscr{C}_{\Gamma}(h)(x)-h(x)\cdot \mathscr{C}_{\Gamma}^*(g)(x)\big ),
 \end{equation}
where $\mathscr{C}_{\Gamma}^*$ is the  adjoint operator to $\mathscr{C}_{\Gamma}$.

\begin{theorem}
\label{t:UchiyamaFactor}
For any $f\in H^1_{b}(\mathbb{R})$ there exist  sequence $\{\lambda_j^k\}_{j,k\in\mathbb{Z}}\in \ell^1$ and functions $g_j^{k}, h_j^{k}\in L^\infty(\mathbb{R})$ with compact supports such that $$f=\sum_{k=1}^\infty\sum_{j=1}^{\infty} \lambda_{j}^{k}\, \Pi_b(g_j^{k}, h_j^{k}).$$  Moreover, we have that:
$$
\left\Vert f\right\Vert_{H^1_{b}(\mathbb{R})} \approx \inf\left\{\sum_{k=1}^\infty \sum_{j=1}^{\infty} \left\vert \lambda_j^k\right\vert \left\Vert g_j^k\right\Vert_{L^2(\mathbb{R})}\left\Vert h_j^k\right\Vert_{L^2(\mathbb{R})}: f=\sum_{k=1}^\infty\sum_{j=1}^{\infty}\lambda_j^k \,\Pi_b(g_j, h_j) \right\}.
$$
\end{theorem}

The commutator $[g,T]$ of a function $g$ and  an operator $T$ denotes the new operator acting on suitable functions $f$ defined by
$ [g,T](f) := g\,T(f) - T(gf)$. It is well known that a function $a\in \bmo(\mathbb{R})$ (respectively $a\in {\rm VMO}(\mathbb{R})$) if and only if $[a,H]$,  the commutator  of $a$ with  $H$ the Hilbert transform,  is a  bounded operator on $L^p(\mathbb{R})$ \cite{CRW} (respectively is a compact operator \cite{U1}). In \cite{LNWW}, functions $a$ in   $\bmo_b(\mathbb{R})$ (respectively in ${\rm VMO}_b(\mathbb{R})$) were characterized via boundedness (respectively compactness) of $[a,\widetilde{\mathscr{C}}_{\Gamma}]$, the commutator with the related Cauchy operator.
To characterize $\bmo_b(\mathbb{R})$ and ${\rm VMO}_b(\mathbb{R})$ we will consider the commutator
of the Cauchy integral not with functions in $\bmo_b(\mathbb{R})$ or  ${\rm VMO}_b(\mathbb{R})$ but with  those functions divided by the accretive function $b$. In other words we will consider for the  next theorems the commutator $[\frak{A}/b, \mathscr{C}_{\Gamma}]$ where $\frak{A}$ is in  ${\rm BMO}_b(\mathbb{R})$ or in ${\rm VMO}_b(\mathbb{R})$. 

\begin{theorem}\label{th upper}
Let  $b(x) = 1+ i A'(x)$, and $p\in (1,\infty)$. 
If  $\frak A\in {\rm BMO}_{b}(\mathbb{R})$, then we have 
$$
\left \|\left[{\frak A / b}, \mathscr{C}_{\Gamma}\right]\right \|_{L^p(\mathbb{R})\to L^p(\mathbb{R})} \lesssim \| \frak A\|_{{\rm BMO}_{b}(\mathbb{R})}.
$$
Conversely, for any complex function $\frak A$ such that ${\frak A / b}$ is a real-valued function and ${\frak A / b}\in L^1_{loc}(\R)$, if $\left\|\left[{\frak A/ b}, \mathscr{C}_{\Gamma}\right]\right\|_{L^p(\mathbb{R})\to L^p(\mathbb{R})}<\infty$,
then
 $\frak A\in {\rm BMO}_{b}(\mathbb{R})$ with 
$$
 \| \frak A\|_{{\rm BMO}_{b}(\mathbb{R})} \lesssim \left\|\left[{\frak A / b}, \mathscr{C}_{\Gamma}\right]\right \|_{L^p(\mathbb{R})\to L^p(\mathbb{R})}.
$$
\end{theorem}

\begin{theorem}\label{compact}
Let $b(x) = 1+ i A'(x)$  and $p\in(1,\infty)$. 
If  $\frak A\in {\rm VMO}_{b}(\mathbb{R})$, then we have 
$
 \left[{\frak A / b}, C_{\Gamma}\right]
$
is compact on $L^p(\R)$.
Conversely, for any complex function  
 $\frak A\in {\rm BMO}_{b}(\mathbb{R})$ such that ${\frak A / b}$ is a real-valued function and ${\frak A / b}\in L^1_{loc}(\R)$, if $
 \left[{\frak A / b}, \mathscr{C}_{\Gamma}\right]
$
is compact on $L^p(\R)$, then 
$\frak A\in {\rm VMO}_{b}(\mathbb{R})$.
\end{theorem}

Note that it is possible to deduce these three theorems as corollaries from the results in \cite{LNWW} and \cite{LW} directly,
as we will show in Section 3.

The paper is organized as follows.   In Section \ref{s:2} we collect the necessary preliminaries needed to explain the result. In Section \ref{s:classical-to-Meyer-spaces} we provide a connection between the classical Hardy and ${\rm BMO}$ spaces and the spaces introduced by Meyer.  In Section \ref{s:factorization} we provide a second proof of Theorem \ref{t:UchiyamaFactor} using a clever construction due to Uchiyama \cite{U}.

We use the standard notation that $A\lesssim B$ to mean that there exists an absolute constant $C$ such that $A\leq CB$; $A\gtrsim B$ has the analogous definition.  Finally, $A\approx B$ if $A\lesssim B$ and $B\lesssim A$.
We  use $\langle f,g\rangle_{L^2(\mathbb{R})}$ to denote the  $L^2$-pairing $\int_{\mathbb{R}} f(x)\,g(x)\, dx$. We denote   $C_c^{\infty}(\mathbb{R})$ the space of compactly supported infinitely differentiable functions on $\mathbb{R}$. We denote by $\chi_I$ the characteristic function of the set $I \subset \mathbb{R}$, defined by $\chi_I(x)=1$ if $x\in I$ and $\chi_I(x)=0$ otherwise.

\section{Preliminaries}
\setcounter{equation}{0}
\label{s:2}
In this section we introduce basic notions of accretive functions;  the classical spaces: Hardy space $H^1(\mathbb{R})$, the space of bounded mean oscillation functions ${\rm BMO}(\mathbb{R})$, and the space of vanishing mean oscillation ${\rm VMO}(\mathbb{R})$; and their counterparts,  the Meyer-type Hardy spaces:  $H^1_b(\mathbb{R})$, ${\rm BMO}_b(\mathbb{R})$, and ${\rm VMO}_b(\mathbb{R})$, for $b$ an accretive function. We also introduce the Cauchy integral operator $ \mathscr{C}_{\Gamma}$ associated to a Lipschitz curve $\Gamma$ and the related Cauchy integral operator $\widetilde{\mathscr{C}}_{\Gamma}$.

 A function   $b:\mathbb{R}\to\mathbb{C}$ is \emph{accretive} if  $b\in L^{\infty}(\mathbb{R})$ and 
${\rm Re}\,b(x)\geq \delta >0$ for all $x\in\mathbb{R}$.

     A locally integrable real-valued function $f: \mathbb{R}
    \rightarrow \mathbb{R}$ is said to be of \emph{bounded mean
    oscillation}, written $f \in \bmo$ or $f \in \bmo(\mathbb{R})$, if   
    \[
        \|f\|_{\bmo}
        := \sup_I\frac{1}{|I|}\int_I |f(x) - f_I| \,dx < \infty.
    \]
    Here the supremum is taken over all intervals
    $I$  in $\mathbb{R}$ and 
   $
        f_I
        := \frac{1}{|I|}\int_I f(y) \,dy
    $
    is the average of the function~$f$ over the interval~$I$.

 A ${\rm BMO}$ function $f: \mathbb{R}
    \rightarrow \mathbb{R}$ is said to be of \emph{vanishing mean
    oscillation}, written $f \in {\rm VMO}$ or $f \in {\rm VMO}(\mathbb{R})$,
    if the following three behaviors occur for small, large and far from the origin intervals respectively,
\begin{itemize}
\item[(i)]  $\quad\quad\;\displaystyle{\lim_{\delta\to 0} \sup_{I:|I|<\delta} \frac{1}{|I|}\int_I |f(x) - f_I| \,dx =0,  }$
\item[(ii)]  $\quad\quad\;\displaystyle{\lim_{R\to \infty} \sup_{I:|I| >R} \frac{1}{|I|}\int_I |f(x) - f_I| \,dx =0, \quad}$ and
\item[(iii)]   $\quad\quad\;\displaystyle{\lim_{R\to \infty} \sup_{I: I\cap I(0,R)=\emptyset} \frac{1}{|I|}\int_I |f(x) - f_I| \,dx =0. }$
\end{itemize}

The \emph{Hardy  space} $H^1(\mathbb{R})$ consists of those integrable functions  $f\, :\, \mathbb{R}\to \mathbb{\R}$ that admit an atomic decomposition $f(x)=\sum_{j=1}^{\infty}\lambda_ja_j(x)$ where $\sum_{j=1}^{\infty} |\lambda_j|<\infty$ and the functions $a_j(x)$ are $L^\infty$-\emph{atoms} (respectively  \emph{$L^2$-atoms}) in the sense that: each atom $a_j$ is supported on an interval $I_j$ and the following $L^{\infty}$-size condition
$\|a_j\|_{L^{\infty}(\mathbb{R})} \leq 1/|I_j|$
(respectively, $L^2$-size condition $\|a_j\|_{L^2(\mathbb{R} )}\leq C|I_j|^{-1/2}$) and  cancellation condition $\int_{\mathbb{R}} a_j(x)\, b(x)\,dx =0$ hold. 
The $H^1$-norm can be defined using either type of atoms, for example
\[\|f\|_{H^1(\mathbb{R})}:= \inf\Big \{\sum_{j=1}^{\infty}|\lambda_j|\, : \, f=\sum_{j=1}^{\infty}\lambda_ja_j, \; \mbox{$a_j$ are $L^2$-atoms}\Big \}.\]
If instead we use $L^{\infty}$-atoms we will get an equivalent norm \cite[Section 6.6.4]{Gra}.
It is well known that ${\rm BMO}$ is the dual of $H^1(\mathbb{R})$ \cite{FS}.

 A function $f:\mathbb{R}\to \mathbb{C}$ is said to be in $H^1_b(\mathbb{R})$, the \emph{Meyer-type Hardy space associated with the accretive function $b$} if and only if 
 $$ bf\in H^1(\mathbb{R}), \quad \mbox{moreover}\;  \|f\|_{H^1_b(\mathbb{R})} := \|bf\|_{H^1(\mathbb{R})}.$$
 In other words, the function $f\in H^1_b(\mathbb{R})$ admits an atomic decomposition $f(x)=\sum_{j=1}^{\infty}\lambda_ja_j(x)$ where $\sum_{j=1}^{\infty} |\lambda_j|<\infty$ and the functions $a_j(x)$ are $L^\infty$-\emph{atoms} (respectively $L^2$-atoms) in the sense that: each $H^1_b$ atom $a_j$ is supported on an interval $I_j$ and the following $L^{\infty}$-size condition
$\|a_j\|_{L^{\infty}(\mathbb{R})} \leq 1/|I_j|$
(respectively, $L^2$-size condition $\|a_j\|_{L^2(\mathbb{R} )}\approx \|ba_j\|_{L^2(\mathbb{R} )}\leq C|I_j|^{-1/2}$), and  cancellation condition  $\int_{\mathbb{R}} a_j(x)\, b(x)\,dx =0$ hold.

    A locally integrable function $\frak A: \mathbb{R}
    \rightarrow \mathbb{C}$ is said to be in  $\bmo_b(\mathbb{R})$, the \emph{Meyer-type ${\rm BMO}$ space associated with the accretive function $b$} ,
    if
    \[
       {\frak A / b}\in \bmo(\mathbb{R}),
    \]
    and we define its norm naturally to be 
    $
       \|\frak A\|_{\bmo_b(\R)}:= \left\|{\frak A / b}\right\|_{\bmo(\R)}.
    $
As a consequence of the $H^1$-${\rm BMO}$ duality ${\rm BMO}_b(\mathbb{R})$ is the dual of $H^1_b(\mathbb{R})$ \cite{Me}.

A locally integrable function $\frak A: \mathbb{R}
    \rightarrow \mathbb{C}$ is said to be in  ${\rm VMO}_b(\mathbb{R})$, the \emph{Meyer-type ${\rm VMO}$ space associated with the accretive function $b$} ,
    if
    \[
       {\frak A / b}\in {\rm VMO}(\mathbb{R}).
    \]
Suppose $\Gamma$ is a curve in the complex plane~$\C$ and~$f$
is a function defined on the curve~$\Gamma$. The \emph{Cauchy
integral} of~$f$ is the operator $\mathcal{C}_{\Gamma}$ defined on the complex plane  for $z\notin\Gamma$
by
\begin{equation}\label{eq.11.0}
      \mathcal{C}_{\Gamma}(f)(z)
    := \frac{1}{2\pi i} \int_{\Gamma}\frac{f(\zeta)}{z - \zeta}\,d\zeta.
\end{equation}
A curve~$\Gamma$ is said to be a \emph{Lipschitz curve} if it
can be written in the form~$\Gamma = \{x + iA(x): x \in \R\}$
where~$A: \R \rightarrow \R$ satisfies a Lipschitz condition
\begin{equation}\label{eq:l1}
  |A(x_1) - A(x_2)| \leq L|x_1 - x_2| \quad \text{for all } x_1, x_2 \in \R.
\end{equation}
The best constant~$L$ in~\eqref{eq:l1} is referred to as the
\emph{Lipschitz constant} of~$\Gamma$ or of~$A(x)$. One can
show that $A$~satisfies a Lipschitz condition if and only
if~$A$ is differentiable almost everywhere on~$\R$ and~$A'\in
L^\infty(\mathbb R)$. The Lipschitz constant is~$L =
\|A'\|_{\infty}$.

The \emph{Cauchy integral associated with the Lipschitz
curve}~$\Gamma$ is the singular integral operator $\mathscr{C}_{\Gamma}$  defined for $x\in\mathbb{R}$ and acting on functions $f\in C_c^{\infty}(\mathbb{R})$ by
\begin{equation}\label{eq:l2}
  \mathscr{C}_{\Gamma}(f)(x)
  := {\rm p.v.} \frac{1}{\pi i}\int_{\R} \frac{(1 + iA'(y))f(y)}{y-x + i(A(y) - A(x))}\,dy,
\end{equation}
where $f \in C_{c}^{\infty}(\R)$. The kernel
of~$\mathscr{C}_{\Gamma}$ is given by
\[{C}_{\Gamma}(x,y) =  \frac{1}{\pi i}\frac{1 + iA'(y)}{y-x + i(A(y) - A(x))}.\]
Note that this is not a standard Calder\'on--Zygmund kernel because the function~$1 + iA'$
does not necessarily possess any smoothness. As noted
in~\cite[p.289]{Gra}, 
 the~$L^p$-boundedness
of~$\mathscr{C}_{\Gamma}$ is equivalent to that of the
related operator $\widetilde{\mathscr{C}}_{\Gamma}$ defined for $x\in\mathbb{R}$ by
\begin{equation}\label{eq:l3}
  \widetilde{\mathscr{C}}_{\Gamma}(f)(x)
  := {\rm p.v.} \frac{1}{\pi i}\int_{\R} \frac{f(y)}{y-x + i(A(y) - A(x))}\,dy.
\end{equation}
Moreover,  the kernel of $\widetilde{\mathscr{C}}_{\Gamma}$ is given by
\begin{equation}\label{eq:Ckernel}
    \widetilde{{C}}_\Gamma(x,y)
    = \frac{1}{\pi i}\frac{1}{y-x + i(A(y) - A(x))}.
\end{equation}
The kernel $C_{\Gamma}(x,y)$ of~$\widetilde{\mathscr{C}}_{\Gamma}$ satisfies standard  size and smoothness\footnote{Namely: (size) $|C_{\Gamma}(x,y)|\lesssim 1/|x-y|$ for all $x,y\in\mathbb{R}$  and (smoothness)  $|C_{\Gamma}(x,y)-C_{\Gamma}(x_0,y) | + |C_{\Gamma}(y,x)-C_{\Gamma}(y,x_0) |\lesssim |x-x_0|/|x-y|^2$ for all $x,x_0,y\in\mathbb{R}$ such that $|x-x_0|\leq |y-x|/2$.} estimates \cite[Lemma 3.3]{LNWW} and is therefore bounded on $L^p(\mathbb{R})$ for $p\in (1,\infty)$. Therefore,
while the operator~$\mathscr{C}_{\Gamma}(f)$  is initially defined for~$f \in
C_{c}^{\infty}(\R)$,  it can be extended
to all~$f \in L^{p}(\R)$, for each~$p \in (1,\infty)$. 

An operator $T$ defined on $L^p(\R)$  is \emph{compact} on $L^p(\R)$ if $T$ maps bounded subsets of $L^p(\R)$ into precompact sets. In other words, for all bounded sets $E\subset L^p(\R)$, $T (E)$ is precompact. A set S is \emph{precompact} if its closure is compact.

\section{From Classical Spaces to  Meyer Hardy spaces }
\setcounter{equation}{0}
\label{s:classical-to-Meyer-spaces}

In this section we take advantage of the known weak factorization result for $H^1(\mathbb{R})$ in terms of the Calder\'on-Zygmund singular integral operator $\widetilde{\mathscr{C}}_{\Gamma}$ as well as the characterization of ${\rm BMO}(\mathbb{R})$ via the boundedness of the commutator with $\widetilde{\mathscr{C}}_{\Gamma}$ and of ${\rm VMO}(\mathbb{R})$ via the compactness of the same commutator \cite{LNWW}.

We first consider the adjoint operator $\mathscr{C}_{\Gamma}^*(g)$.
By a direct calculation, we can verify that for $f,g\in L^2(\R)$,
\begin{align*}
\langle  \mathscr{C}_{\Gamma}(f), g\rangle_{L^2(\mathbb{R})} &=\int_{\R }  {\rm p.v.} \frac{1}{\pi i}\int_{\R} \frac{(1 + iA'(y))f(y)}{y-x + i(A(y) - A(x))}\,dy\ g(x)dx\\
&= \int_{\R }  {\rm p.v.} \frac{1}{\pi i}\int_{\R} \frac{1}{y-x + i(A(y) - A(x))} g(x)\, dx\   (1 + iA'(y))f(y) \,dy\\
&= \int_{\R }   b(y)\,(\widetilde{\mathscr{C}}_{\Gamma})^*( g)(y) \,  f(y) \,dy\, = \, \langle f, \mathscr{C}_{\Gamma}^*(g)\rangle_{L^2(\mathbb{R})}.
\end{align*}
We therefore conclude that
\begin{equation}\label{adjoint-C_Gamma}
\mathscr{C}_{\Gamma}^*(g)(x) = b(x)\cdot {(\widetilde{\mathscr{C}}_{\Gamma})^*}( g)(x).  
\end{equation}
Note that $(\widetilde{\mathscr{C}}_{\Gamma})^*=-\widetilde{\mathscr{C}}_{\Gamma}$.

We now use the weak factorization for $H^1(\mathbb{R})$  --valid for  $m$-linear Calder\'on-Zygmund operators \cite[Theorem 1.3]{LW}-- for the particular Calder\'on-Zygmund operator $\widetilde{\mathscr{C}}_{\Gamma}$ \cite{LNWW}, to obtain the desired weak factorization for the Meyer-type Hardy space $H_b^1(\mathbb{R})$. 
\begin{proof}[First Proof of Theorem~\ref{t:UchiyamaFactor}] The function $f\in H^1_b(\mathbb{R})$ if and only if $bf\in H^1(\mathbb{R})$ but by weak factorization of $H^1(\mathbb{R})$ there are a sequence $\{\lambda_{s,k}\}_{s,k\in\mathbb{Z}}$ and compactly supported bounded functions $G^k_s$ and $H^k_s$ such that $b\, f=\sum_{k=1}^{\infty}\sum_{s=1}^{\infty} \lambda_{s,k}\, \Pi(G^k_s,H^k_s)$. 
Where the bilinear form $\Pi(G,H)$ is defined by
\[ \Pi(G,H)(x) = G(x)\cdot  \widetilde{\mathscr{C}}_{\Gamma}(H)(x)-H(x)\cdot (\widetilde{\mathscr{C}}_{\Gamma})^*(G)(x).\]
Moreover, 
$$\|b\,f\|_{H^1(\mathbb{R})} \approx\inf \Big \{ \sum_{k=1}^{\infty}\sum_{s=1}^{\infty} |\lambda_{s,k}| \|G^k_s\|_{L^2(\mathbb{R})}\|H^k_s\|_{L^2(\mathbb{R})}: b\, f=\sum_{k=1}^{\infty}\sum_{s=1}^{\infty} \lambda_{s,k}\, \Pi(G^k_s,H^k_s)\Big \}.$$
Therefore
$$ f (x) = \sum_{k=1}^{\infty}\sum_{s=1}^{\infty} \lambda_{s,k} \, \frac{1}{b}\, \Pi(G^k_s,H^k_s) (x)
  =  \sum_{k=1}^{\infty}\sum_{s=1}^{\infty} \lambda_{s,k}\,  \Pi_b\big (G^k_s,\frac{H^k_s}{b}\big )(x).
$$
The last identity since by definition \eqref{bilinear-form} of the bilinear form $\Pi_b(g,h)$, the fact that $\mathscr{C}_{\Gamma}(f)=\widetilde{\mathscr{C}}_{\Gamma}(bf)$, and identity~\eqref{adjoint-C_Gamma}, we have  that
\begin{eqnarray}\label{Pib-Pi}
\frac{1}{b}\,\Pi(G,H)(x) 
                            & = & \Pi_b \big (G, \frac{H}{b}\big )(x) .
 \end{eqnarray}
 Let $g^k_s:=G^k_s$ and $h^k_s:=H^k_s/b$, both are compactly supported bounded functions and
 $$f(x)  =  \sum_{k=1}^{\infty}\sum_{s=1}^{\infty} \lambda_{s,k}\, \Pi_b\big (g^k_s, h^k_s\big )(x).$$
 Moreover $\|f\|_{H^1_b(\mathbb{R})}= \|b\,f\|_{H^1(\mathbb{R})} $, therefore
 $$\|f\|_{H^1_b(\mathbb{R})} \approx\inf \Big \{ \sum_{k=1}^{\infty}\sum_{s=1}^{\infty} |\lambda_{s,k}| \|g^k_s\|_{L^2(\mathbb{R})}\|h^k_s\|_{L^2(\mathbb{R})}\, : \, f=\sum_{k=1}^{\infty}\sum_{s=1}^{\infty} \lambda_{s,k}\, \Pi_b(g^k_s,h^k_s)\Big \}.$$
The last identity because $g^k_s=G^k_s$ and $\|h^k_s\|_{L^2(\mathbb{R})} \approx \|b\, h^k_s\|_{L^2(\mathbb{R})} = \|H^k_s\|_{L^2(\mathbb{R})} $ since $b$ is an accretive function.
 This proves Theorem~\ref{t:UchiyamaFactor}.
  \end{proof}

If we know how to construct the functions $G^k_s$ and $H^k_s$ then we know how to construct 
 the functions $g^k_s$ and $h^k_s$, and viceversa. In the next section we provide an explicit construction of the functions $g^k_s$ and $h^k_s$,
 following Uchiyama's blueprint directly in our setting.
 
 Before proceeding, we provide proofs of Theorem~{\rm\ref{th upper}} and of Theorem~\ref{compact} relying on the corresponding results for the related Cauchy integral operator $\widetilde{\mathscr{C}}_{\Gamma}$. Namely, $a\in {\rm BMO}$  (respectively in ${\rm VMO}$) if and only if $[a,\widetilde{\mathscr{C}}_{\Gamma}]$ is bounded on $L^p(\mathbb{R})$ (respectively, is  compact in $L^p(\mathbb{R})$) for $p\in (1,\infty)$. Furthermore,  the following norm comparability holds 
 $$\|a\|_{{\rm BMO}} \approx \| [a,\widetilde{\mathscr{C}}_{\Gamma}]\|_{ L^p(\mathbb{R})\to L^p(\mathbb{R})}.$$
 
\begin{proof}[Proof of Theorem~{\ref{th upper}}]
For $b(x)=1+ i A'(x)$, suppose $\frak A$ is in $\bmo_b(\R)$, that is $\frak{A}/b\in \bmo(\mathbb{R})$; a direct calculation, using that $\mathscr{C}_{\Gamma}(g)=\widetilde{\mathscr{C}}_{\Gamma}(bg)$,  shows that
\begin{align*}
\left[ {\frak A/ b}, \mathscr{C}_\Gamma \right](f)(x)
&=\left[ {\frak A/ b}, \widetilde{\mathscr{C}}_\Gamma \right](bf)(x).  
\end{align*}
Thus, since by \cite[Theorem 1.1]{LNWW} the commutator $[\frak A/b,\widetilde{\mathscr{C}}_{\Gamma}]$ is bounded on $L^p(\mathbb{R})$, we get 
\begin{align*}\left\|\left[ {\frak A/ b}, \mathscr{C}_\Gamma \right](f)\right\|_{L^p(\R)}&=\| [ {\frak A/ b}, \widetilde {\mathscr{C}}_\Gamma ](bf) \|_{L^p(\R)}\leq \| [ {\frak A/ b}, \widetilde{\mathscr{C}}_\Gamma ]\|_{L^p(\R)\to L^p(\R)}\|bf\|_{L^p(\mathbb{R})}\\
&\lesssim \left \| \frak A/b\right \|_{\bmo(\mathbb{R})} \|f\|_{L^p(\mathbb{R})} \, = \, C\left \| \frak A\right \|_{\bmo_b(\mathbb{R})} \|f\|_{L^p(\mathbb{R})}.
\end{align*}

Conversely, for any given complex function $\frak A$ such that ${\frak A /b}$ is a real-valued function, ${\frak A / b}\in L^1_{loc}(\R)$ and $\left\|\left[{\frak A / b}, \mathscr{C}_{\Gamma}\right] \right\|_{ L^p(\mathbb{R})\to L^p(\mathbb{R})}<\infty$, we see that
\begin{align*}
\big\|\big[ {\frak A/ b}, \widetilde{\mathscr{C}}_\Gamma \big](f)\big\|_{L^p(\R)} &= \left\|\left[ {\frak A/ b}, \mathscr{C}_\Gamma \right]\left( {f/ b}\right)\right\|_{L^p(\R)}\leq \left\|\left[{\frak A / b}, \mathscr{C}_{\Gamma}\right]\right\|_{L^p(\mathbb{R})\to L^p(\mathbb{R})} \left\|{f/b}\right\|_{L^p(\R)}\\
&\lesssim \left\|\left[{\frak A / b}, \mathscr{C}_{\Gamma}\right] \right\|_{L^p(\mathbb{R})\to L^p(\mathbb{R})} \left\|{f}\right\|_{L^p(\R)}.
\end{align*} 
Hence, the commutator $\big [ {\frak A/ b}, \widetilde{\mathscr{C}}_\Gamma  \big ] $ is bounded on $L^p(\mathbb{R})$ and by \cite[Theorem 1.1]{LNWW} we conclude that
 ${\frak A/ b}$ is in $\bmo(\R)$ and 
$$ \left\|{\frak A/ b}\right\|_{\bmo(\R)}\lesssim  \big\|\big[ {\frak A/b}, \widetilde{\mathscr{C}}_\Gamma \big]\big \|_{L^p(\mathbb{R})\to L^p(\mathbb{R})} \lesssim \left\|\left[{\frak A / b}, \mathscr{C}_{\Gamma}\right]\right \|_{L^p(\mathbb{R})\to L^p(\mathbb{R})}.$$
Hence, we conclude that  ${\frak A}$ is in $\bmo_b(\R)$ and 
$$ \left\|{\frak A}\right\|_{\bmo_b(\R)}\lesssim   \left\|\left[{\frak A /b}, \mathscr{C}_{\Gamma}\right]\right\|_{L^p(\mathbb{R})\to L^p(\mathbb{R})}.$$
This finishes the proof of Theorem~\ref{th upper}.
\end{proof}

Similar considerations yield the proof of Theorem~\ref{compact} from the knowledge that $a\in{\rm VMO}$ if and only if $[a,\widetilde{\mathscr{C}}_{\Gamma}]$ is a compact operator on $L^p(\mathbb{R})\to L^p(\mathbb{R})$ \cite[Theorem 1.2]{LNWW}.
\begin{proof}[Proof Theorem~\ref{compact}]
For $b(x)=1+ i A'(x)\in L^{\infty}(\R)$, suppose $\frak A$ is in ${\rm VMO}_b(\R)$, that is $\frak{A}/b\in {\rm VMO}(\mathbb{R})$. 
Therefore by \cite[Theorem 1.2]{LNWW} the commutator $[\frak{A}/b, \widetilde{\mathscr{C}}_{\Gamma}]$ is compact.
Let $E$ be a bounded subset of $L^p(\R)$, then  $bE$ is  a bounded subset of $L^p(\R )$ since
\[ \sup_{g\in bE} \|g\|_{L^p(\R)}=\sup_{f\in E} \| bf\|_{L^p(\R)} \leq \|b\|_{L^{\infty}}\sup_{f\in E} \| f\|_{L^p(\R)} < \infty.\] 
Therefore $[\frak{A}/b, \widetilde{\mathscr{C}}_{\Gamma}] (bE) $ is a precompact set. 
 Recall that 
$\left[ {\frak A/ b}, \widetilde{\mathscr{C}}_\Gamma \right](bf)(x) = \left[ {\frak A/ b}, \mathscr{C}_\Gamma \right](f)(x)$ for all $f\in L^p(\R)$.
Thus  
$$[\frak{A}/b,\mathscr{C}_{\Gamma}](E) = [\frak{A}/b, \widetilde{\mathscr{C}}_{\Gamma}] (bE).$$
Hence $[\frak{A}/b,\mathscr{C}_{\Gamma}](E) $ is a precompact set for all given bounded subsets $E$ of $L^p(\R)$. By definition $[\frak{A}/b,\mathscr{C}_{\Gamma}]$ is compact.

Conversely, suppose  $[\frak{A}/b,\mathscr{C}_{\Gamma}]$ is compact. Then given a bounded subset $F$ of $L^p(\R)$, $F/b$ is also a bounded subset of $L^p(\R)$ since $\|b\|_{L^{\infty}}\geq 1$. Therefore  $[\frak{A}/b,\mathscr{C}_{\Gamma}] (F/b)$ is precompact, but as before, 
$$[\frak{A}/b,\mathscr{C}_{\Gamma}] (F/b) = [\frak{A}/b,\widetilde{\mathscr{C}}_{\Gamma}] (F).$$
Thus $[\frak{A}/b,\widetilde{\mathscr{C}}_{\Gamma}] (F)$ is a precompact set for all bounded subsets $F$ of $L^p(\R)$.
By definition $[\frak{A}/b,\widetilde{\mathscr{C}}_{\Gamma}]$ is  a compact operator in $L^p(\R )$  and by \cite[Theorem 1.2]{LNWW} we conclude that $\frak{A}/b \in {\rm VMO}(\mathbb{R})$, and therefore $\frak{A}\in {\rm VMO}_b(\mathbb{R})$. This finishes the proof of Theorem~\ref{compact}.
\end{proof}

\section{Weak Factorization of the Meyer Hardy space - Uchiyama's construction}
\setcounter{equation}{0}
\label{s:factorization}

In this section we present a constructive proof of  functions $g^k_s$ and $h^k_s$ for $k,s\geq 1$, appearing  in the weak factorization of $H^1_b(\mathbb{R})$. This argument follows Uchiyama's procedure closely \cite{U}.

\subsection{The upper bound in Theorem~\ref{t:UchiyamaFactor}}

Given a function $f\in H^1_b(\mathbb{R})$, suppose we have a factorization of the form $f =  \sum_{k=1}^{\infty}\sum_{s=1}^{\infty} \lambda_{s,k}\, \Pi_b\big (g^k_s, h^k_s\big )$ with  $\{\lambda_{s,k}\}\in \ell^1$ and $g^k_s$ and $h^k_s$ compactly supported and bounded functions, as claimed in Theorem~\ref{t:UchiyamaFactor}. Then it suffices to verify the following Lemma~\ref{t:upperbound} to conclude that
\[ \|f\|_{H_b^1(\mathbb{R})} \leq \sum_{k=1}^{\infty}\sum_{s=1}^{\infty} |\lambda_{s,k}| \|g^k_s\|_{L^2(\mathbb{R})}\|h^k_s\|_{L^2(\mathbb{R})}.\]
 \begin{lemma}
\label{t:upperbound}
Let $g,h\in L^\infty(\mathbb{R})$ with compact supports.  Then $\Pi_b(g,h)$ is in $H^1_b(\R)$ with
$$
\left\Vert \Pi_b(g,h)\right\Vert_{H^1_{b}(\mathbb{R})}\lesssim  \left\Vert g\right\Vert_{L^2(\mathbb{R})}\left\Vert h\right\Vert_{L^2(\mathbb{R})}.
$$
\end{lemma}

\begin{proof}

We first point out that for any $g,h\in L^\infty(\mathbb{R})$ with compact supports,
$\Pi_b(g,h)$ is compactly supported in ${\rm supp} (g)\cup {\rm supp} (h)$. 
Next, it is easy to see  that $\Pi_b(g,h)\in L^2(\mathbb{R})$, using that $\mathscr{C}_{\Gamma}$ is a bounded operator in $L^2(\mathbb{R})$, indeed,
\begin{eqnarray*}
 \|\Pi_b(g,h)\|_{L^2(\R)} & \lesssim  &
 \|g\|_{L^{\infty}(\mathbb{R})}\|h\|_{L^2(\mathbb{R})} +\|h\|_{L^{\infty}(\mathbb{R})}\|g\|_{L^2(\mathbb{R})} 
\end{eqnarray*}
Moreover, since by definition of adjoint $\langle h,  \mathscr{C}_{\Gamma}^*(g)\rangle_{L^2(\mathbb{R})} = \langle \mathscr{C}_{\Gamma}(h), g\rangle_{L^2(\mathbb{R})}$, the following cancellation holds,
$$ \int_{\R}  \Pi_b(g,h)(x)\, b(x)\,dx = \int_{\R}\big(g(x)\cdot \mathscr{C}_{\Gamma}(h)(x)-h(x)\cdot \mathscr{C}_{\Gamma}^*(g)(x)\big) \, dx =0.$$
Hence, it is clear that up to a multiplication by certain constant, the bilinear form
$\Pi_b(g,h)(x)$ is a  $L^2$-atom of $H^1_b(\R)$, that is, 
$\Pi_b(g,h)\in H^1_b(\R)$.

Now it suffices to verify the $H^1_b(\R)$ norm of $\Pi_b(g,h)$ is controlled by an absolute multiple of $\left\Vert g\right\Vert_{L^2(\mathbb{R})}\left\Vert h\right\Vert_{L^2(\mathbb{R})}$. A simple duality computation shows for $\frak A\in {\rm BMO}_{b}(\mathbb{R})$ and for any $g,h\in L^\infty(\mathbb R^n)$ with compact supports:
\begin{eqnarray*}
\left\langle \frak A, \Pi_b(g,h)\right\rangle_{L^2(\mathbb{R})}  =  \left\langle {\frak A/ b}, g\cdot \mathscr{C}_{\Gamma}(h)-h\cdot \mathscr{C}_{\Gamma}^*(g)\right\rangle_{L^2(\mathbb{R})} = \left\langle g,\left[{\frak A/ b}, \mathscr{C}_{\Gamma}\right](h)\right\rangle_{L^2(\mathbb{R})}.
\end{eqnarray*}
Remember that  $\langle f,g\rangle_{L^2(\mathbb{R})}$ denotes the  $L^2$ pairing $\int_{\mathbb{R}} f(x)\,g(x)\, dx$, not the $L^2$ inner product.
Thus,  from the upper bound as in Theorem \ref{th upper}, we obtain that
\begin{eqnarray*}
|\left\langle \frak A, \Pi_b(g,h)\right\rangle_{L^2(\mathbb{R})}| = \left|\left\langle g,\left[{\frak A/ b},\mathscr{C}_{\Gamma}\right](h)\right\rangle_{L^2(\mathbb{R})}\right| \leq C \|\frak A\|_{{\rm BMO}_{b}(\mathbb{R})  }\|g\|_{L^2(\mathbb R)} \|h\|_{L^2(\mathbb R)}.
\end{eqnarray*}
This, together with the duality result of \cite{Me},  $H^1_{b}(\mathbb{R})^{*}={\rm BMO}_{b}(\mathbb{R})$, shows that
\begin{eqnarray*}
\left\Vert \Pi_b(g,h)\right\Vert_{H^1_{b}(\mathbb{R})} & \approx &  \sup_{\left\Vert \frak A\right\Vert_{{\rm BMO}_{b}(\mathbb{R})}\leq 1}  \left\vert \left\langle \frak A,\Pi_b(g,h)\right\rangle_{L^2(\mathbb{R})}\right\vert\\
& \lesssim &  \left\Vert g\right\Vert_{L^2(\mathbb{R})}\left\Vert h\right\Vert_{L^2(\mathbb{R})} \sup_{\left\Vert \frak A\right\Vert_{{\rm BMO}_{b}(\mathbb{R})}\leq 1} \left\Vert \frak A\right\Vert_{{\rm BMO}_{b}(\mathbb{R})}\\
& \lesssim & \left\Vert g\right\Vert_{L^2(\mathbb{R})}\left\Vert h\right\Vert_{L^2(\mathbb{R})}.
\end{eqnarray*}

\end{proof}

\subsection{The factorization and the lower bound in Theorem~\ref{t:UchiyamaFactor}}

The proof of the factorization and of the lower bound  in Theorem~\ref{t:UchiyamaFactor} is more algorithmic in nature and follows a proof strategy pioneered by Uchiyama in \cite{U}.  We begin with a fact that will play a prominent role in the algorithm below.  It is a modification of a related fact for the standard Hardy space $H^1(\mathbb{R})$.

\begin{lemma}\label{lemma Hardy}
Let $b(x)=1+iA'(x)$  with $A'\in L^{\infty}(\mathbb{R})$.
Suppose $f$ is a function satisfying:  $\int_{\mathbb{R}} f(x)\,b(x)\,dx=0$, and $|f(x)|\leq \chi_{I(x_0,1)}(x)+\chi_{I(y_0,1)}(x)$, where $|x_0-y_0|:=M>100$ and $I(z_0,L):=\{z\in\mathbb{R}\, :\, |z-z_0|<L\}$. Then $f\in H^1_b(\mathbb{R})$ and we have
\begin{align}
 \|f\|_{H_{b}^1(\mathbb{R})} \lesssim \log M.
\end{align}
\end{lemma}

The lemma when $b\equiv 1$ is stated in \cite[Lemma 2.2]{LW} without proof, the authors refer the reader to  \cite[Lemma 3.1]{DLWY} and \cite[Lemma 4.3]{LW2} where the corresponding lemma, in the Bessel and Neumann Laplacian settings respectively, is stated and proved. We can not apply directly  \cite[Lemma 2.2]{LW} because although $F=bf$ will satisfy $\int_{\mathbb{R}}F(x)\, dx=0$ by hypothesis,  it will not satisfy that $|F(x)|\leq \chi_{I(x_0,1)}(x)+\chi_{I(y_0,1)}(x)$, instead it will satisfy $|F(x)|\leq |b(x)|\big (\chi_{I(x_0,1)}(x)+\chi_{I(y_0,1)}(x)\big )$.  But we could apply it to 
$F_0=bf/\|b\|_{L^{\infty}(\mathbb{R})}$, since $|F_0(x)|\leq \chi_{I(x_0,1)}(x)+\chi_{I(y_0,1)}(x)$, to conclude that $F_0\in H^1(\mathbb{R})$  and $\|F_0\|_{H^1(\mathbb{R})} \lesssim \log M$. Finally since $\|F_0\|_{H^1(\mathbb{R})}= \|bf\|_{H^1(\mathbb{R})} /\|b\|_{L^{\infty}(\mathbb{R})}$ we conclude that  $f\in H^1_b(\mathbb{R})$ 
and 
$$\|f\|_{H^1_b(\mathbb{R})}=\|bf\|_{H^1(\mathbb{R})}\lesssim \| b\|_{L^{\infty}(\mathbb{R})}\log M\lesssim \log M.$$
Nevertheless, for completeness, we present here a direct construction of an atomic decomposition in $H^1_b(\mathbb{R})$ for $f$ that yields the estimate claimed in Lemma~\ref{lemma Hardy} that could have an interest in itself, it also provides a proof for \cite[Lemma 2.2]{LW} by setting $b\equiv 1$. This construction yields an atomic decomposition for  $f = \sum_{j\in\mathbb{Z}} \lambda_j a_j$.
However the $H^1_b$  $L^{\infty}$-atoms $a_j$ built in the proof of Lemma~\ref{lemma Hardy} for the specific given  $f$ are not the $H^1$ $L^{\infty}$-atoms one would get by multiplying by $\|b\|_{\infty}/b$ the $H^1$ $L^{\infty}$-atoms $A_j$ obtained by the same procedure   applied to $F_0$ when  $b\equiv 1$.

\begin{proof}[Proof of Lemma~\ref{lemma Hardy} ]
Suppose $f$ satisfies the conditions as stated in the lemma above. We will show by  construction that
$f$ has an atomic decomposition with respect to the $H_b^1(\R)$ $L^{\infty}$-atoms, using an idea from Coifman \cite{CW}.   To see this, we first define two functions $f_1(x) $ and $f_2(x)$ by
\[  f_1(x)= f(x) \, \chi_{I(x_0,1)}(x) \quad\quad\mbox{and} \quad\quad f_2(x)= f(x) \, \chi_{I(y_0,1)}(x) .\]
Then we have $f=f_1+f_2$  and by hypothesis and definition
$$ |f_1(x)| \lesssim   \chi_{I(x_0,1)}(x) \quad {\rm and}\quad |f_2(x)| \lesssim   \chi_{I(y_0,1)}(x). $$

Define  
\begin{align*}
g_1^1(x)&:=\frac{\chi_{I(x_0,2)}(x)}{ \int_{I(x_0,2)}b(z)\, dz }\int_{\mathbb{R}}f_1(y)\,b(y) \, dy,\\ 
f_1^1(x)&:= f_1(x)- g_1^1(x),\\
\alpha_1^1&:=\|f_1^1\|_\infty |I(x_0,2)|.
\end{align*}
Then we claim that $a_1^1:= (\alpha_1^1)^{-1} f_1^1$  is an $H^1_b(\R)$ $L^\infty$-atom.
First, by definition  $a_1^1$ is supported on $I(x_0,2)$. Moreover, we have that
\begin{align*}
  \int_{\R} a_1^1(x) \, b(x)\,dx& =(\alpha_1^1)^{-1}\int_{\R} \left(f_1(x)- g_1^1(x) \right) b(x)\,dx\\
  &= (\alpha_1^1)^{-1}\bigg(\int_{\R} f_1(x)\, b(x)\, dx- \int_{\R}\frac{\chi_{I(x_0,2)}(x)}{ \int_{I(x_0,2)}\,b(z)\,dz }\, b(x) \,dx  \int_{\R} f_1(y)\, b(y)\, dy\bigg)\\
  &= (\alpha_1^1)^{-1}\bigg(\int_{\R} f_1(x)\, b(x)\, dx-   \int_{\R} f_1(y)\, b(y)\, dy\bigg)\, = \, 0
\end{align*}
and that
\begin{align*}
\|a_1^1\|_\infty \leq |(\alpha_1^1)^{-1}|\|f_1^1\|_\infty = \frac1{|I(x_0,2)|}.
\end{align*}
Thus, $a_1^1$  is an $H^1_b(\R)$ $L^\infty$-atom. We also have the following  estimate for the coefficient~$\alpha_1^1$.
\begin{align*}
|\alpha_1^1|&=\|f_1^1\|_\infty |I(x_0,2)| \le \|f_1\|_\infty |I(x_0,2)| + \|g_1^1\|_\infty |I(x_0,2)| \\
&\leq  |I(x_0,2)|+ \frac{ |I(x_0,2)|}
{ \big|\int_{I(x_0,2)}b(z)\,dz\big| } \int_{\mathbb{R}}|f_1(y)|\,|b(y)| \, dy\\
& \leq 4+2\|b\|_{L^{\infty}(\mathbb{R})} \, \leq  \, 6\|b\|_{L^{\infty}(\mathbb{R})} \, \lesssim  \,1.
\end{align*}
Here we used the facts that $\|b\|_{L^{\infty}(\mathbb{\R})} < \infty$,   $f_1\leq \chi_{I(x_0,1)}$,  $|I(x_0,L)|=2L$, and 
$$ \Big|\int_{I(x_0,2)}b(z)\,dz\Big|\geq \Big|\int_{I(x_0,2)}{\rm Re}\, b(z)\,dz\Big|\geq |I(x_0,2)|. $$
Moreover, we see that
\begin{align*}
f_1(x)= f_1^1(x)+ g_1^1(x)= \alpha_1^1 a_1^1(x)+ g_1^1(x).
\end{align*}
For $g_1^1(x)$, we further write it as
\begin{align*}
g_1^1(x)= \big (g_1^1(x)- g_1^2(x)\big )+ g_1^2(x)=:f_1^2(x)+g_1^2(x)
\end{align*}
with
$$g_1^2(x):=\frac{\chi_{I(x_0,4)}(x)}{\int_{I(x_0,4)}b(z)\,dz} \int_{\mathbb{R}}f_1(y)\,b(y) \, dy.$$  
Again, we define
\begin{align*}
\alpha_1^2&:=\|f_1^2\|_\infty |I(x_0,4)| \quad{\rm and}\quad a_1^2:= (\alpha_1^2)^{-1} f_1^2,
\end{align*}
and following similar estimates as for $a_1^1$, we see that
$a_1^2$  satisfies the compact support condition and the size condition
$ \|a_1^2\|_\infty \leq \frac{1}{|I(x_0,4)|}$. Hence, it suffice to see that
it also satisfies the cancellation condition with respect to $b$. In fact,
\begin{align*}
  \int_{\R} a_1^2(x)\, b(x)\,dx
  & =(\alpha_1^2)^{-1}\int_{\R} \left( g_1^1(x) -  g_1^2(x) \right) b(x)\,dx\\
  &= (\alpha_1^2)^{-1}\bigg(  \int_{\R}\frac{\chi_{I(x_0,2)}(x)}{ \int_{I(x_0,2)}b(z)\,dz } \,b(x)\,dx  \int_{\R} f_1(y)\,b(y)\,dy \\
  &\hskip 1in -\int_{\R}\frac{\chi_{I(x_0,4)}(x)}{ \int_{I(x_0,4)}b(z)\,dz } \,b(x)\,dx  \int_{\R} f_1(y)\,b(y)\,dy \bigg)\\
  &= (\alpha_1^2)^{-1}\bigg(\int_{\R} f_1(y)\,b(y)\,dy-   \int_{\R} f_1(y)\,b(y)\,dy\bigg)
  =0.
\end{align*}
As 
a consequence,
we see that $a_1^2$
is an $H^1_b(\R)$ $L^\infty$-atom.
Moreover, we have the following estimate for the coefficient $a_1^2$.
\begin{align*}
|\alpha_1^2|&=\|f_1^2\|_\infty |I(x_0,4)| \le \|g_1^1\|_\infty |I(x_0,4)| + \|g_1^2\|_\infty |I(x_0,4)| \\
&\leq   \frac{ | I(x_0,4)|}{ \big|\int_{I(x_0,2)}b(z)\,dz\big| } \int_{\mathbb{R}}|f_1(y)|\,|b(y)| \, dy 
+ \frac{| I(x_0,4)|}{ \big|\int_{I(x_0,4)}b(z)\,dz\big| } \int_{\mathbb{R}}|f_1(y)|\,|b(y)| \, dy \\ 
&  \leq 4\|b\|_{L^{\infty}(\mathbb{R})}+2\|b\|_{L^{\infty}(\mathbb{R})} \, = \, 6\|b\|_{L^{\infty}(\mathbb{R})}\, \lesssim \,1.
\end{align*}
Here again we use the fact that for every $L>0$,
$$ \Big|\int_{I(x_0,L)}b(z)\,dz\Big|\geq \Big|\int_{I(x_0,L)}{\rm Re} \,b(z)\,dz\Big|\geq |I(x_0,L)|. $$

Then we have
\begin{align*}
f_1(x)=\sum_{i=1}^2 \alpha_1^i a_1^i(x)+ g_1^2(x).
\end{align*}
Continuing in this fashion  we see that for $i \in \{1, 2, . . . , i_0\}$,
\begin{align*}
f_1(x)=\sum_{i=1}^{i_0} \alpha_1^i a_1^i(x)+ g_1^{i_0}(x),
\end{align*}
where for $i \in \{2, . . . , i_0\}$,
\begin{align*}
g_1^{i}(x)&:=\frac{\chi_{I(x_0,2^i)}(x)}{\int_{I(x_0,2^i)}b(z)dz} \int_{\mathbb{R}}f_1(y)\, b(y)\, dy,\\    
f_1^{i}(x)&:= g_1^{i-1}(x)-g_1^{i}(x),\\
\alpha_1^i&:=\|f_1^i\|_\infty | I(x_0,2^i)| \quad{\rm and}\\
a_1^i(x)&:= (\alpha_1^i)^{-1} f_1^i(x).
\end{align*}
Here we choose $i_0$ to be the smallest positive integer such that $ I(y_0,1)\subset I(x_0,2^{i_0})$. Then from the condition that $|x_0-y_0|=M$, we obtain that
$$ i_0\approx \log_2 M. $$

Moreover, for $i \in \{1, 2, . . . , i_0\}$, we have the estimate of the coefficients as follows.
\begin{align*}
|\alpha_1^i|& \leq  6\|b\|_{L^{\infty}(\mathbb{R})} \, \lesssim  \,1.
\end{align*}

Following the same steps, we also obtain that  for $i \in \{1, 2, . . . , i_0\}$,
\begin{align*}
f_2(x)=\sum_{i=1}^{i_0} \alpha_2^i a_2^i(x)+ g_2^{i_0}(x),
\end{align*}
where for $i \in \{2, . . . , i_0\}$,
\begin{align*}
g_2^{i}(x)&:={\chi_{I(y_0,2^i)}(x)\over \int_{I(y_0,2^i)}b(z)\, dz } \int_{\mathbb{R}}f_2(y)\,b(y) \, dy,\\   
f_2^{i}(x)&:= g_2^{i-1}(x)-g_2^{i}(x),\\
\alpha_2^i&:=\|f_2^i\|_\infty |I(y_0,2^i)| \quad{\rm and}\\
a_2^i(x)&:= (\alpha_2^i)^{-1} f_2^i(x).
\end{align*}
Similarly, for $i \in \{1, 2, . . . , i_0\}$, we can verify that each $a_2^i$ is an $H^1_b(\R)$ $L^\infty$-atom and the coefficient satisfies
\begin{align*}
|\alpha_2^i|\lesssim {1}.
\end{align*}

Combining the decompositions above, we obtain that
\begin{align*}
f(x)=\sum_{j=1}^2\bigg(\sum_{i=1}^{i_0} \alpha_j^i a_j^i(x)+ g_j^{i_0}(x)\bigg).
\end{align*}
We now consider the tail $g_1^{i_0}(x) + g_2^{i_0}(x)$. To handle that, consider
the interval $\overline I$  centered at the point
$  {x_0+y_{0}\over 2}$
with sidelength $2^{i_0+1}$. Then, it is clear that $I(x_0,1)\cup I(y_0,1) \subset \overline I$,
and that $I(x_0,2^{i_0}), I(y_0,2^{i_0}) \subset \overline I$.
Thus, since by hypothesis $\int_{\mathbb{R}} f(y)\,b(y)\,dy=0$, we get that 
$${\chi_{\overline I}(x)\over  \int_{\overline I}b(z)\, dz }\int_{I(x_0,1)}f_1(y)\,b(y)\,dy+{\chi_{\overline I}(x)\over \int_{\overline I}b(z)\,dz }\int_{I(y_0,1)}f_2(y)\,b(y)\,dy=0.$$
Hence, we write
\begin{align*}
g_1^{i_0}(x) + g_2^{i_0}(x)&=\bigg(g_1^{i_0}(x) -
{\chi_{\overline I}(x)\over  \int_{\overline I}b(z)\,dz }\int_{I(x_0,1)}f_1(y)\,b(y)\,dy\bigg) \\
&\quad+ \bigg( g_2^{i_0}(x) -
{\chi_{\overline I}(x)\over  \int_{\overline I}b(z)\,dz }\int_{I(y_0,1)}f_2(y)\,b(y)\,dy\bigg)\\
 &=: f_1^{i_0+1}(x)+f_2^{i_0+1}(x).
\end{align*}
For $j=1,2$, we now define 
\begin{align*}
\alpha_j^{i_0+1}&:=\|f_j^{i_0+1}\|_\infty |\overline{I}| \quad{\rm and}\\  
a_j^{i_0+1}(x)&:= (\alpha_j^{i_0+1})^{-1} f_j^{i_0+1}(x).
\end{align*}
Again we can verify that for $j=1,2$, $a_j^{i_0+1}$ is an $H^1_b(\R)$ $L^\infty$-atom supported in $I$ with
the appropriate size and cancellation conditions
$$\|a_j^{i_0+1}\|_\infty \leq {1 /  |\overline I |} \quad\quad \mbox{and} \quad\quad \int_{\mathbb{R}} a_j^{i_0+1}(x)\,b(x)\, dx =0.$$
Moreover, we also have 
$$|\alpha_j^{i_0+1}| \lesssim {1 }.$$

Thus, we obtain an atomic decomposition for $f$ 
\begin{align*}
f(x)=\sum_{j=1}^2\sum_{i=1}^{i_0+1} \alpha_j^i a_j^i(x),
\end{align*}
which implies that $f\in H^1_b(\mathbb{R})$ and
\begin{align*}
\|f\|_{H^1_b(\mathbb{R})} &\leq \sum_{j=1}^2\sum_{i=1}^{i_0+1}  |\alpha_j^i| \lesssim  \sum_{j=1}^2\sum_{i=1}^{i_0+1}  {1}
\lesssim {\log M}.
\end{align*}
This  finishes the proof of Lemma \ref{lemma Hardy}.
\end{proof}

Repeating the proof  we get that if $r>0$, $\int_{\mathbb{R}}f(x)\, b(x)\, dx=0$ and $|f(x)| \leq \chi_{I(x_0,r)}(x) + \chi_{I(y_0,r)}(x)$
where $|x_0-y_0|\geq rM$, then  $f\in H^1_b(\mathbb{R})$ and
\begin{equation}\label{rescaled-lem-Hardy}
 \| f \|_{H^1_b(\mathbb{R})} \lesssim  r \, \log{M}. 
 \end{equation}
The additional $r$ comes from the estimates of the coefficients $|\alpha^i_j |\lesssim r$ for $i=1,\dots, i_0+1$ and $j=1,2$ where $i_0\sim \log M$.

Ideally, given an $H^1_{b}(\mathbb{R})$-atom $a$, we would like to find $g,h\in L^2(\mathbb{R})$ such that $\Pi_b(g,h)=a$ pointwise.  While this can not be accomplished in general, the theorem below shows that it is ``almost'' true.

\begin{theorem}
\label{thm:ApproxFactorization}
For every $H^1_{b}(\mathbb{R})$ $L^\infty$-atom $a(x)$ and for all $\varepsilon>0$ there exist
a large positive number $M$ and $g,h\in L^\infty(\mathbb{R})$ with compact supports such that:
$$
\left\Vert a-\Pi_b(h,g)\right\Vert_{H^1_{b}(\mathbb{R})}<\varepsilon
$$
and $\left\Vert g\right\Vert_{L^2(\mathbb{R})}\left\Vert h\right\Vert_{L^2(\mathbb{R})}\lesssim  M$.
\end{theorem}

\begin{proof}

Let $a(x)$ be an $H^1_{b}(\mathbb{R})$ $L^\infty$-atom, supported in $I(x_0,r)$, the interval centred at $x_0$ with radius $r$.  We first consider the construction of the explicit bilinear form $\Pi_b(h,g)$  and the approximation 
to $a(x)$.  To begin with, 
fix $\varepsilon>0$.  Choose $M\in [100,\infty)$ sufficiently large so that $$ M^{-1}\,{\log M} <\varepsilon.$$ Now select $y_0\in\mathbb{R}$
such that  $ y_0-x_0={Mr}$.
For this $y_0$ and  for any $y\in I(y_0,r)$ and any $x\in I(x_0,r)$, we  have
$|x-y|>{Mr/2}$.  We set
\begin{align}\label{gh}
 g(x):=\chi_{I(y_0,r)}(x)\quad{\rm and}\quad h(x):= -\frac{a(x)}{(\widetilde{\mathscr{C}}_{\Gamma})^*(g)(x_0)}.  
\end{align}
We first note that
\begin{align}\label{claim degenerate}
\left| (\widetilde{\mathscr{C}}_{\Gamma})^*(g)(x_0) \right|\gtrsim \, M^{-1}.
\end{align}
In fact, from the expression of $(\widetilde{\mathscr{C}}_{\Gamma})^*(g)(x_0)=-\widetilde{\mathscr{C}}_{\Gamma}(g)(x_0)$ 
we have that
\begin{align*}
| (\widetilde{\mathscr{C}}_{\Gamma})^* (g)(x_0)| &=   \left|  \frac{1}{\pi i}\int_{ I(y_0,r)} \frac{1}{y-x_0 + i(A(y) - A(x_0))}\,dy\right| \gtrsim {\, M^{-1}}.
\end{align*}
As a consequence, we get that the claim \eqref{claim degenerate} holds.

From the definitions of the functions $g$ and $h$, we obtain that $\operatorname{supp}(g)=I(y_0,r)$ and $\operatorname{supp}(h)=I(x_0,r)$. Moreover, from \eqref{claim degenerate}  and the size estimate for the atom, we obtain that
$$\|g\|_{L^\infty(\mathbb{R})} \approx 1\quad\textnormal{ and }\quad \|h\|_{L^\infty(\mathbb{R})} =  \frac{1}{|(\widetilde{\mathscr{C}}_{\Gamma})^*(g)(x_0)|} \|a\|_{L^\infty(\mathbb{R})}\lesssim  M r^{-1}. $$
And we also get that
$$\|g\|_{L^2(\mathbb{R})} \approx r^{1/2}\quad\textnormal{ and }\quad \|h\|_{L^2(\mathbb{R})} =  \frac{1}{|(\widetilde{\mathscr{C}}_{\Gamma})^*(g)(x_0)|} \|a\|_{L^2(\mathbb{R})}\lesssim M r^{-1/2}. $$
Hence $\|g\|_{L^2(\mathbb{R})} \|h\|_{L^2(\mathbb{R})} \lesssim M$.  Now write
\begin{align*}
a(x)-\Pi_b(h,g)(x)&=a(x)-{1\over b(x)}\big(g(x)\cdot \mathscr{C}_{\Gamma}(h)(x)-h(x)\cdot \mathscr{C}_{\Gamma}^*(g)(x)\big)\\
&= \Big(a(x)+ {1\over b(x)}h(x)\cdot \mathscr{C}_{\Gamma}^*(g)(x) \Big) - {1\over b(x)} g(x) \cdot \mathscr{C}_{\Gamma}(h)(x)\\
&=: W_1(x)+W_2(x).
\end{align*}

We first turn to $W_1(x)$. By definition and using equation~\eqref{adjoint-C_Gamma}, we have that
\begin{align*}
W_1(x)&=a(x)+ {1\over b(x)}\Big(-\frac{a(x)}{(\widetilde{\mathscr{C}}_{\Gamma})^*(g)(x_0)}\Big)\cdot \big(b(x)\cdot (\widetilde{\mathscr{C}}_{\Gamma})^*( g)(x)\big)\\
&=a(x)\Big[1 -   \frac{(\widetilde{\mathscr{C}}_{\Gamma})^*( g)(x)}{(\widetilde{\mathscr{C}}_{\Gamma})^*(g)(x_0)}\Big] \\
& = a(x) \cdot  \frac{(\widetilde{\mathscr{C}}_{\Gamma})^*(g)(x_0)-(\widetilde{\mathscr{C}}_{\Gamma})^*( g)(x)}{(\widetilde{\mathscr{C}}_{\Gamma})^*(g)(x_0)}.
\end{align*}
Thus, since $(\widetilde{\mathscr{C}}_{\Gamma})^*=-\widetilde{\mathscr{C}}_{\Gamma}$, we get that for 
every $x\in I(x_0,r)$,
\begin{align*}
|W_1(x)| &= |a(x)| \cdot  \frac{|\widetilde{\mathscr{C}}_{\Gamma}(g)(x_0)-\widetilde{\mathscr{C}}_{\Gamma}( g)(x)|}{|\widetilde{\mathscr{C}}_{\Gamma}(g)(x_0)|}\\
& \leq  C  M \|a\|_{L^\infty(\mathbb{R})}  \int_{I(y_0,r)} \frac{|x-x_0|}{|x-y|^{2} } \,dy\\
&\leq  C  M r^{-1}\, r \,{r}\,{(Mr)^{-2} }  
\, = \, C\,(M r)^{-1}.
\end{align*}
Here we used the standard smoothness estimate for the Calder\'on-Zygmund kernel $\widetilde{C}_{\Gamma}(x,y)$ of $\widetilde{\mathscr{C}}_{\Gamma}$, see \cite[Lemma 3.3.]{LNWW} or \cite[Example 4.1.6]{Gra}.
Since it is clear that $W_1(x)$ is supported in $I(x_0,r)$,
we obtain that $$ |W_1(x)|\leq  C \, (M r)^{-1} \chi_{I(x_0,r)}(x).$$

We next estimate $W_2(x)$. By definition, it is clear that $W_2(x)$ is supported in $I(y_0,r)$, and we have
\begin{align*}
W_2(x) &= {1\over b(x)}\, \chi_{I(y_0,r)}(x) \cdot \mathscr{C}_{\Gamma}\Big(-\frac{a(\cdot)}{(\widetilde{\mathscr{C}}_{\Gamma})^*(g)(x_0)}\Big)(x)\\
 &= -{1\over b(x)}\, \chi_{I(y_0,r)}(x)\, \frac{1}{(\widetilde{\mathscr{C}}_{\Gamma})^*(g)(x_0)} \cdot  \mathscr{C}_{\Gamma}(a(\cdot))(x)\\
 &= -{1\over b(x)}\, \chi_{I(y_0,r)}(x)\, \frac{1}{(\widetilde{\mathscr{C}}_{\Gamma})^*(g)(x_0)}\   \frac{1}{\pi i}\int_{I(x_0,r)} \frac{(1 + iA'(y))\, a(y)}{y-x + i(A(y) - A(x))}\,dy\\
  &= -{1\over b(x)}\, \chi_{I(y_0,r)}(x)\, \frac{1}{(\widetilde{\mathscr{C}}_{\Gamma})^*(g)(x_0)}\   \frac{1}{\pi i}\int_{I(x_0,r)} \widetilde{\mathscr{C}}_{\Gamma}(x,y) \,{ b(y)\, a(y)}\,dy\\
  &= -{1\over b(x)} \, \chi_{I(y_0,r)}(x)\, \frac{1}{(\widetilde{\mathscr{C}}_{\Gamma})^*(g)(x_0)}\   \frac{1}{\pi i}\int_{I(x_0,r)} \Big(\widetilde{\mathscr{C}}_{\Gamma}(x,y)-\widetilde{\mathscr{C}}_{\Gamma}(x,x_0)\Big) { b(y)\, a(y)}\,dy.
\end{align*}
Here the last equality follows from the cancellation condition of the $H^1_b(\R)$ $L^\infty$-atom~$a(x)$.
Hence, we have for $x\in I(y_0,r)$ (otherwise $W_2(x)=0$ and any estimate will hold)
\begin{align*}
|W_2(x)| &\leq {1\over |b(x)|}\,\chi_{I(y_0,r)}(x)\,\frac{1}{|(\widetilde{\mathscr{C}}_{\Gamma})^*(g)(x_0)|}\   \frac{1}{\pi }\int_{I(x_0,r)} \big|\widetilde{C}_{\Gamma}(x,y)-\widetilde{C}_{\Gamma}(x,x_0)\big|\,  |b(y)|\, |a(y)|\,dy\\
 &\lesssim \,\chi_{I(y_0,r)}(x) \,M  \int_{I(x_0,r)}\|a\|_{L^\infty(\mathbb{R})} \frac{ |y-x_0|}{|x-x_0|^{2} } \,dy\\
&\lesssim (M r)^{-1}  \, \chi_{I(y_0,r)}(x).
\end{align*}
Once again using the smoothness of the kernel $\widetilde{C}_{\Gamma}(x,y)$ of $\widetilde{\mathscr{C}}_{\Gamma}$. 

Combining the estimates of $W_1$ and $W_2$, we obtain that
\begin{align}\label{size}
 \Big|a(x)-\Pi_b(g,h)(x)\Big|\lesssim  {(M r)^{-1}}(\chi_{I(x_0,r)}(x)+\chi_{I(y_0,r)}(x)).
\end{align}
Next we point out that
\begin{align}\label{cancellation R}
\int_{\mathbb{R}} \Big[a(x)-\Pi_b(g,h)(x) \Big] b(x)\, dx= 0,
\end{align}
since $a(x)$ has cancellation with respect to $b(x)$ and the  same holds for $\Pi_b(g,h)(x)$.

Then the size estimate \eqref{size} and the cancellation \eqref{cancellation R},  together with the result in Lemma \ref{lemma Hardy},  more specifically estimate \eqref{rescaled-lem-Hardy}, imply that $a-\Pi_b(g,h)\in H^1_b(\mathbb{R})$ and 
$$ \big\|a-\Pi_b(g,h)\big\|_{H^1_{b}(\mathbb{R})} \lesssim \, M^{-1}\log M <C\epsilon. $$
This proves the result.
\end{proof}

We deduce from the theorem the following corollary concerning $H^1(\mathbb{R})$ $L^{\infty}$-atoms.
\begin{cor} 
For every $H^1(\mathbb{R})$ $L^\infty$-atom $A(x)$ and for all $\varepsilon>0$ there exist
$M>0$ and compactly supported $L^\infty$ functions  $G$ and $H$ such that
$
\|A-\Pi(H,G)\|_{H^1(\mathbb{R})}<\varepsilon
$
and $\| G\|_{L^2(\mathbb{R})}\|H\|_{L^2(\mathbb{R})}\lesssim  M$.
\end{cor}
\begin{proof}
Note that if $A$ is an $H^1(\mathbb{R})$ $L^{\infty}$-atom then $A/b$ is an $H^1_b(\mathbb{R})$ $L^{\infty}$-atom, hence by Theorem \ref{thm:ApproxFactorization}  for all $\varepsilon>0$ there are  $M>0$ and compactly supported $L^{\infty}$ functions $g,h$ such that $\| A/b - \Pi_b(g,h)\|_{H^1_b(\mathbb{R})} \lesssim \varepsilon$. By \eqref{Pib-Pi} $\Pi_b(g,h)=b\,\Pi(g,bh)$, this implies $\|A - b\,\Pi_b(g,h)\|_{H^1(\mathbb{R})}=\|A-\Pi(g,bh)\| \lesssim \varepsilon$. Let $G=g$ and $H=bh$,
these are compactly supported $L^{\infty}$ functions, furthermore $\| G\|_{L^2(\mathbb{R})}\|H\|_{L^2(\mathbb{R})}\approx \| g\|_{L^2(\mathbb{R})}\|h\|_{L^2(\mathbb{R})}\lesssim  M$.
\end{proof}

With this approximation result, we can now prove the main result. 
\begin{proof}[Constructive Proof of Theorem \ref{t:UchiyamaFactor}] 
By Theorem \ref{t:upperbound} we have that $\left\Vert \Pi_b(g,h)\right\Vert_{H^1_{b}(\mathbb{R})}\lesssim  \left\Vert g\right\Vert_{L^2(\mathbb{R})} \left\Vert h\right\Vert_{L^2(\mathbb{R})}$. It follows that if $f\in H^1_b(\mathbb{R})$  then  for any representation of the form $$f=\sum_{k=1}^\infty\sum_{j=1}^\infty\lambda_j^k\, \Pi_b(g_j^{k}, h_j^{k})$$ we have that
$$
\left\Vert f\right\Vert_{H^1_{b}(\mathbb{R})} \lesssim  \sum_{k=1}^\infty\sum_{j=1}^{\infty}\left\vert \lambda_j^k\right\vert \left\Vert g_j^k\right\Vert_{L^2(\mathbb{R})}\left\Vert h_j^k\right\Vert_{L^2(\mathbb{R})}.
$$
Consequently,
$$
\left\Vert f\right\Vert_{H^1_{b}(\mathbb{R})} \lesssim  \inf\left\{\sum_{k=1}^\infty\sum_{j=1}^{\infty}\left\vert \lambda_j^k\right\vert \left\Vert g_j^k\right\Vert_{L^2(\mathbb{R})}\left\Vert h_j^k\right\Vert_{L^2(\mathbb{R})}: f=\sum_{k=1}^\infty\sum_{j=1}^{\infty}\lambda_j^k \,\Pi_b(g_j^k, h_j^k) \right\}.
$$

We turn to show that the other inequality holds and that it is possible to obtain such a decomposition for any $f\in H^1_{b}(\mathbb{R}^n)$.  By the definition of $H^1_{b}(\mathbb{R})$,  for any $f\in H^1_{b}(\mathbb{R})$ we can find a sequence $\{\lambda_{j}^{1}\}\in \ell^1$ and sequence of $H^1_{b}(\mathbb{R})$ $L^{\infty}$-atoms $a_j^{1}$ so that $f=\sum_{j=1}^{\infty} \lambda_j^{1} a_{j}^{1}$ and $\sum_{j=1}^{\infty} \left\vert \lambda_j^{1}\right\vert \leq C_0 \left\Vert f\right\Vert_{H^1_b(\mathbb{R})}$.

We explicitly track the implied absolute constant $C_0$ appearing from the atomic decomposition since it will play a role in the convergence of the approach.  Fix $\varepsilon>0$ so that $\varepsilon C_0<1$. Then we also have a large positive number $M$ with ${M^{-1}\log M }<\epsilon$.  We apply Theorem \ref{thm:ApproxFactorization} to each atom $a_{j}^{1}$.  So there exists $g_j^{1}, h_j^{1}\in L^\infty(\mathbb{R}^n)$ with compact supports and satisfying  $\left\Vert g_j^{1}\right\Vert_{L^2(\mathbb{R})}\left\Vert h_j^{1}\right\Vert_{L^2(\mathbb{R})}\lesssim M $ and
\begin{equation*}
\left\Vert a_j^{1}-\Pi_b(g_j^{1}, h_j^{1})\right\Vert_{H^1_{b}(\mathbb{R})}<\varepsilon\quad {\rm for\ all\ } j>0.
\end{equation*}
Now note that we have
\begin{eqnarray*}
f & = & \sum_{j=1}^{\infty} \lambda_{j}^{1} a_j^{1}=   \sum_{j=1}^{\infty} \lambda_{j}^{1} \,\Pi_b(g_j^{1}, h_j^{1})+\sum_{j=1}^{\infty} \lambda_{j}^{1} \left(a_j^{1}-\Pi_b(g_j^{1}, h_j^{1})\right) :=  M_1+E_1.
\end{eqnarray*}
Observe that we have
\begin{eqnarray*}
\left\Vert E_1\right\Vert_{H^1_{b}(\mathbb{R})} & \leq & \sum_{j=1}^{\infty} \left\vert \lambda_j^{1}\right\vert \left\Vert a_j^{1}-\Pi_b(g_j^{1}, h_j^{1})\right\Vert_{H^1_{b}(\mathbb{R})} \leq  \varepsilon \sum_{j=1}^{\infty} \left\vert \lambda_j^{1}\right\vert \leq \varepsilon C_0\left\Vert f\right\Vert_{H^1_{b}(\mathbb{R})}.
\end{eqnarray*}
We now iterate the construction on the function $E_1$.  Since $E_1\in H^1_{b}(\mathbb{R})$, we can apply the atomic decomposition in $H^1_{b}(\mathbb{R})$ to find a sequence $\{\lambda_j^{2}\}\in \ell^1$ and a sequence of $H^1_{b}(\mathbb{R})$  $L^{\infty}$-atoms $\{a_j^{2}\}$ so that $E_1=\sum_{j=1}^{\infty} \lambda_j^{2} a_{j}^{2}$ and
$$
\sum_{j=1}^{\infty} \left\vert \lambda_j^{2}\right\vert \leq C_0 \left\Vert E_1\right\Vert_{H^1_b(\mathbb{R})}\leq \varepsilon C_0^2 \left\Vert f\right\Vert_{H^1_{b}(\mathbb{R})}.
$$

Again, we will apply Theorem \ref{thm:ApproxFactorization} to each $L^{\infty}$-atom $a_{j}^{2}$.  So there exist $g_j^{2}, h_j^{2}\in L^\infty(\mathbb{R})$ with compact supports and satisfying  $\left\Vert g_j^{2}\right\Vert_{L^2(\mathbb{R})}\left\Vert h_j^{2}\right\Vert_{L^2(\mathbb{R})}\lesssim  M$ and
\begin{equation*}
\left\Vert a_j^{2}-\Pi_b(g_j^{2}, h_j^{2})\right\Vert_{H^1_{b}(\mathbb{R})}<\varepsilon \quad \mbox{for all  $j>0$}.
\end{equation*}
We then have that:
\begin{eqnarray*}
E_1 & = & \sum_{j=1}^{\infty} \lambda_{j}^{2} a_j^{2}=   \sum_{j=1}^{\infty} \lambda_{j}^{2}\, \Pi_b(g_j^{2}, h_j^{2})+\sum_{j=1}^{\infty} \lambda_{j}^{2} \left(a_j^{2}-\Pi_b(g_j^{2}, h_j^{2})\right) :=  M_2+E_2.
\end{eqnarray*}
But, as before observe that
\begin{eqnarray*}
\left\Vert E_2\right\Vert_{H^1_b(\mathbb{R})} & \leq & \sum_{j=1}^{\infty} \left\vert \lambda_j^{2}\right\vert \left\Vert a_j^{2}-\Pi_b(g_j^{2}, h_j^{2})\right\Vert_{H^1_{b}(\mathbb{R})} \leq  \varepsilon \sum_{j=1}^{\infty} \left\vert \lambda_j^{2}\right\vert \leq \left(\varepsilon C_0\right)^{2}\left\Vert f\right\Vert_{H^1_{b}(\mathbb{R})}.
\end{eqnarray*}
And, this implies for $f$ that we have:
\begin{eqnarray*}
f & = & \sum_{j=1}^{\infty} \lambda_{j}^{1} a_j^{1} =   \sum_{j=1}^{\infty} \lambda_{j}^{1} \,\Pi_b(g_j^{1}, h_j^{1})+\sum_{j=1}^{\infty} \lambda_{j}^{1} \left(a_j^{1}-\Pi_b(g_j^{1}, h_j^{1})\right)\\
 & = & M_1+E_1=M_1+M_2+E_2 =  \sum_{k=1}^{2} \sum_{j=1}^{\infty} \lambda_{j}^{k} \,\Pi_b(g_j^{k}, h_j^{k})+E_2.
\end{eqnarray*}

Repeating this construction for each $1\leq k\leq K$ produces functions $g_j^{k}, h_j^{k}\in L^\infty(\mathbb{R})$ with compact supports and  satisfying $\left\Vert g_j^{k}\right\Vert_{L^2(\mathbb{R})}\left\Vert h_j^{k}\right\Vert_{L^2(\mathbb{R})}\lesssim  M $ for all $j>0$, sequences $\{\lambda_{j}^{k}\}_{j>0}\in \ell^1$ with $\left\Vert \{\lambda_{j}^{k}\}_{j>0}\right\Vert_{\ell^1}\leq \varepsilon^{k-1} C_0^k \left\Vert f\right\Vert_{H^1_{b}(\mathbb{R})}$, and a function $E_K\in H^1_{b}(\mathbb{R})$ with $\left\Vert E_K\right\Vert_{H^1_{b}(\mathbb{R})}\leq \left(\varepsilon C_0\right)^{K}\left\Vert f\right\Vert_{H^1_{b}(\mathbb{R})}$ so that
$$
f=\sum_{k=1}^{K} \sum_{j=1}^\infty \lambda_{j}^{k} \,\Pi_b(g_j^{k}, h_j^{k})+E_K.
$$
Passing $K\to\infty$ gives the desired decomposition of $$f=\sum_{k=1}^{\infty} \sum_{j=1}^\infty \lambda_{j}^{k}\, \Pi_b(g_j^{k}, h_j^{k}).  $$
We also have that:
$$
\sum_{k=1}^{\infty} \sum_{j=1}^\infty\left\vert \lambda_{j}^{k}\right\vert \leq \sum_{k=1}^{\infty} \varepsilon^{-1} (\varepsilon C_0)^{k} \left\Vert f\right\Vert_{H^1_{b}(\mathbb{R})}= \frac{ C_0}{1-\varepsilon C_0}\left\Vert f\right\Vert_{H^1_{b}(\mathbb{R})}.
$$
Therefore $\{\lambda_j^k\}_{j,k\in\mathbb{Z}}$ is in $\ell^1$ as claimed. This finishes the proof of Theorem~\ref{t:UchiyamaFactor}.
\end{proof}
The weak-factorization given by Theorem~\ref{t:UchiyamaFactor}  can be used to prove the lower bound of Theorem~\ref{th upper}, the same way it is done in  for example \cite{LW}.  However we used the upper bound of 
Theorem~\ref{th upper} to prove Lemma~\ref{t:upperbound} responsible for the upper bound in Theorem~\ref{t:UchiyamaFactor}.

\bigskip

{\bf Acknowledgments:} 
 J. Li is supported by ARC DP 160100153 and Macquarie University New Staff Grant.
B. D. Wick's research supported in part by National Science Foundation DMS grants \#1560955 and \#1800057.

\end{document}